\documentclass{amsart}

\usepackage{amssymb,amsbsy,amsthm,amsmath,graphicx,epsfig,color}

\newtheorem{thm}{Theorem}[section]
\newtheorem*{thm*}{Theorem}

\newtheorem{lemma}[thm]{Lemma}

\newtheorem{corollary}[thm]{Corollary}

\newtheorem{ques}[thm]{Question}
\theoremstyle{remark}

\newcommand{\id}{\mathrm{id}}

\newcommand{\N}{\mathbb{N}}

\renewcommand{\P}{\mathbb{P}}

\newcommand{\C}{\mathbb{C}}

\newcommand{\Z}{\mathbf{Z}}

\newcommand{\veps}{\varepsilon}

\newcommand{\aveN}{\frac{1}{N}\sum_{n=1}^N}
\newcommand{\avepN}{\frac{1}{\pi(N)}\sum_{p\in\P,p\leq N}}
\newcommand{\avepM}{\frac{1}{\pi(M)}\sum_{p\in\P,p\leq M}}

\renewcommand{\hat}[1]{\widehat{#1}}

\newcommand{\aveWN}{\frac{1}{WN}\sum_{n=1}^{WN}}
\newcommand{\aveWM}{\frac{1}{WM}\sum_{n=1}^{WM}}
\newcommand{\avek}{\frac{1}{k}\sum_{n=1}^k}

\title{%Everywhere convergence of
Nilsystems and ergodic averages along primes}

\author{Tanja Eisner}
\address{Institute of Mathematics, University of Leipzig,
P.O. Box 100 920, 04009 Leipzig, Germany}
\email{eisner@math.uni-leipzig.de}


\begin{document}

\maketitle

\centerline{\emph{Dedicated to Vitaly Bergelson on the occasion of his  65$^\text{th}$ birthday}}

\begin{abstract}
 A celebrated result by Bourgain and Wierdl states that ergodic averages along primes converge almost everywhere for $L^p$-functions,  $p>1$, with a polynomial version by Wierdl and  Nair. Using an anti-correlation result for the von Mangoldt function due to Green and Tao we observe everywhere convergence of such averages for nilsystems and continuous functions.
\end{abstract}

\section{Introduction}

Ergodic theorems, originally motivated by physics, have found applications in and connections to many areas of mathematics. A prominent example is the 
%following celebrated 
result on almost everywhere convergence of ergodic averages along primes by Bourgain \cite{B86,B} (for $p>\frac{1+\sqrt{3}}{2}$) and subsequently Wierdl \cite{W} (for all $p>1$).

\begin{thm}\label{thm:bourgain}
Let $(X,\mu,T)$ be a measure-preserving system, $p>1$ and $f\in L^p(X,\mu)$. Then the ergodic averages along primes 
\begin{equation}\label{eq:ave-primes}
\avepN T^pf
\end{equation}
converge almost everywhere.
\end{thm}

The proof is based on the Carleson transference principle to the discrete model $(\Z, \text{Shift})$, the Hardy-Littlewood circle method and estimates of prime number exponential sums.  An analogous result for polynomials instead of primes was proved by Bourgain \cite{B88,B}, see also Thouvenot \cite{Th}, with variation estimates by Krause \cite{K} showing that the averages converge rapidly. For analogous estimates for ave\-rages (\ref{eq:ave-primes}) see Zorin-Kranich \cite{ZK-primes}. Moreover, Theorem \ref{thm:bourgain} has been generalised to polynomials of primes by Wierdl \cite{W-diss} and Nair \cite{N1,N}.

Since the proof of Bourgain and Wierdl does not give any information on the set of points where the convergence holds, the following natural question arises.
%is interesting and natural.
%
\begin{ques}\label{ques} For which systems and functions do the ergodic averages along primes (\ref{eq:ave-primes}) converge \emph{everywhere}?
\end{ques}
%
%Recall that the classical Birkhoff ergodic averages converge everywhere for every uniquely ergodic function and every continuous function on it. 

We give a partial answer to this question and show that ergodic averages along polynomials of primes converge everywhere for all nilsystems and all continuous functions. 
For the definition of a nilsystem and a polynomial sequence see Section \ref{sec:W}.

\begin{thm}\label{thm:main}
Let $G/ \Gamma$ be a nilmanifold, $g:\N\to G$ be a polynomial sequence and $F\in C(G / \Gamma)$. Then the averages
$$
\avepN F(g(p)x)
$$
converge for every $x\in G/ \Gamma$. 
Moreover, if $G$ is connected and simply connected, $g(n)=g^n$ 
and the system $(G/ \Gamma,\mu,g)$ is ergodic, then the limit equals $\int_{G/\Gamma} F d\mu$. 
\end{thm}
The key to this result is the powerful theory developed by Green and Tao \cite{GT10,GT12,GT-nil}, partially together with Ziegler \cite{GTZ}, in their study of arithmetic progressions and linear equations in the primes, in particular the asymptotic orthogonality of the modified von Mangoldt function to nilsequences, see Theorem \ref{thm:GT} below. 

Note that nilsystems and nilsequences has been playing a fundamental role in 
%multiple ergodic theorems, %
the study of other kinds of ergodic averages, 
namely the norm convergence of (linear and polynomial) multiple ergodic averages, motivated by Furstenberg's ergodic theoretic proof \cite{Fu} of Szemer\'edi's theorem \cite{Sz} on the existence of arithmetic progressions in large sets of integers. Here is a list of relevant works: Conze, Lesigne \cite{CL}, Furstenberg, Weiss \cite{FW}, Host, Kra \cite{HK02}, Lesigne \cite{Le}, Ziegler \cite{Z-nil}, Host, Kra \cite{HK}, Ziegler \cite{Z}, Bergelson, Host, Kra \cite{BHK}, Bergelson, Leibman, Lesigne \cite{BLL}, Bergelson, Leibman \cite{BL-gen-pol}, Leibman \cite{L-pol,L-corr}, Frantzikinakis \cite{F-corr}, Host, Kra \cite{HK09}, Chu \cite{C}, Eisner, Zorin-Kranich \cite{EZ}, Zorin-Kranich \cite{ZK}. For other applications of the Green-Tao-Ziegler theory to ergodic theorems see, e.g., Frantzikinakis, Host, Kra \cite{FHK07,FHK}, Wooley, Ziegler \cite{WZ}, Bergelson, Leibman, Ziegler \cite{BLZ}, Frantzikinakis, Host \cite{FH}. 

Our argument is similar to (but simpler than) the one in Wooley, Ziegler \cite{WZ} in the context of the norm convergence of multiple polynomial ergodic averages along primes.

\textbf{Acknowledgement.} The author thanks Vitaly Bergelson for correcting the  refe\-rences and is deeply grateful to the referee for careful reading and  suggestions which have considerably improved the paper.

\section{Preliminaries and the $W$-trick}\label{sec:W}

%We mostly follow the terminology of \cite{GT10}.

Let $G$ be an $s$-step Lie group and $\Gamma$ be a discrete cocompact subgroup of $G$. The homogeneous space $G/\Gamma$ together with the Haar measure $\mu$ is called an \emph{$s$-step nilmanifold}. For every $g\in G$, the left multiplication by $g$ is an invertible $\mu$-preserving transformation on $G/\Gamma$, and the triple $(G/\Gamma,\mu,g)$ is called a \emph{nilsystem}. Nilsystems enjoy remarkable algebraic and ergodic properties making them an important class of systems in the classical ergodic theory, see Auslander, Green, Hahn \cite{AGH}, Green \cite{Gr}, Parry \cite{P1,P2} and Leibman \cite{L}.  For example, single and multiple ergodic averages converge everywhere for such systems and continuous functions.

For a continuous function $F$ on $G/\Gamma$ and $x\in G/\Gamma$ the sequence $(F(g^nx))_{n\in\N}$ is called a \emph{(basic linear) nilsequence} as introduced by Bergelson, Host, Kra \cite{BHK}. A nilsequence in their definition is a uniform limit of basic nilsequences (being allowed to come from different systems and functions). Note that the property of Ces\`aro convergence along primes is preserved by uniform limits, so Theorem \ref{thm:main} implies in particular that every nilsequence is Ces\`aro convergent along primes.

Rather than linear sequences $(g^n)$, following Leibman \cite{L}, Green, Tao \cite{GT-nil} and Green, Tao, Ziegler \cite{GTZ}, we will consider polynomial sequences $(g(n))$, where $g:\N\to G$ is called a \emph{polynomial sequence} if it is of the form $g(n)=g_1^{p_1(n)}\cdot\ldots\cdot g_m^{p_m(n)}$ for some $m\in\N$, $g_1,\ldots,g_m\in G$ and some integer polynomials $p_1,\ldots,p_m$. For an abstract equivalent definition see \cite{GT-nil}. A sequence of the form $(F(g(n)x))$ for a continuous function $F$ on $G/\Gamma$ is called a \emph{polynomial nilsequence}. Although this notion seems to be more general than the one of linear basic nilsequences, it is not, see the references at the beginning of the proof of Theorem \ref{thm:main} in the following section.

Note that a nilsequence does not determine $G$, $\Gamma$, $F$ etc.~uniquely, giving room for reductions. For example, we can assume without loss of generality that $x=\id_G\Gamma$. Moreover, denoting by $G^0$ the connected component of the identity in $G$, since we are only interested in the orbit of $x$ under $g(n)$, we can assume without loss of generality that $G=\langle G^0,g_1,\ldots,g_m \rangle$.
%Moreover, passing to the universal cover if necessary, see Leibman \cite[1.11]{L}, we can assume that the connected component $G^0$ of the identity is simply connected.

%Following \cite{GT10}, we fix a smooth metric $d_{G/\Gamma}$ on $G/\Gamma$ and say that $F:G/\Gamma\to\C$ is Lipschitz with Lipschitz constant $M$ if it is so with respect to this metric. In this case we call corresponding nilsequences Lipschitz with Lipschitz constant $M$.

We use the notations $o_{a,b}(1)$ and $O_{a,b}(1)$ to denote a function which converges to zero %as $N\to \infty$
 or is bounded, respectively, for fixed parameters $a,b$ uniformly in all other parameters.

We now introduce the $W$-trick as in Green and Tao \cite{GT10}. Consider
$$
\Lambda'(n):=\begin{cases}
\log n &\quad \text{if } n\in\P,
\\
0 &\quad \text{otherwise}.
\end{cases}
$$
For $\omega\in\N$ define
$$
W=W_\omega:=\prod_{p\in\P,p\leq\omega}p
$$
and for $r<W$ coprime to $W$ define the modified $\Lambda'$-function by
$$
\Lambda'_{r,\omega}(n):=\frac{\phi(W)}{W}\Lambda'(Wn+r), \quad n\in\N,
$$
where $\phi$ denotes the Euler totient function.

The key to our result is the following anti-correlation property of  $\Lambda'_{r,\omega}$ with nilsequences due to Green and Tao \cite{GT10} conditional to the ``M\"obius and nilsequences conjecture'' proven by them later in \cite{GT12}. Here, $\omega:\N\to\N$ is an arbitrary function with $\lim_{N\to\infty}\omega(N)=\infty$ satisfying $\omega(N)\leq \frac{1}{2}\log\log N$ for all large $N\in\N$.
Note that the corresponding function $W:\N\to\N$ is
then $O(\log^{1/2}N)$.

\begin{thm}\label{thm:GT}\emph{(Green-Tao \cite[Prop. 10.2]{GT10})}
Let $\omega(\cdot)$ and $W(\cdot)$ be as above, $G/\Gamma$ be an s-step nilmanifold with a smooth metric, $G$ being connected and simply connected, and let  $(F(g^n x))$ be a bounded
nilsequence on $G/\Gamma$ with Lipschitz constant $M$. Then
$$
\max_{r<W(N),(r,W(N))=1} \left|\aveN (\Lambda'_{r,\omega(N)}(n)-1) F(g^n x)\right|=o_{M,G/\Gamma,s}(1)
$$
as $N\to \infty$.
\end{thm}
An immediate corollary is the following, cf.~Theorem 2.2 (and the discussion afterwards) in Frantzikinakis, Host, Kra \cite{FHK}.  Here, $\omega$ and $W=W_\omega$ are again numbers, not functions.

\begin{corollary}\label{cor}
Let $G/\Gamma$ be an s-step nilmanifold with a smooth metric,  $G$ being connected and simply connected, and let  $(F(g^n x))$ be a bounded nilsequence on $G/\Gamma$ with  Lipschitz constant $M$. Then
$$
\lim_{\omega\to\infty}\limsup_{N\to\infty}\max_{r<W,(r,W)=1} \left|\aveN (\Lambda'_{r,\omega}(n)-1) F(g^n x)\right|=0,
%o_{M,G/\Gamma,s}(1),
$$
where 
%one first takes $\limsup_{N\to\infty}$ and then $\lim_{\omega\to\infty}$.
the convergence is uniform in $F$, $g$ and $x$.
\end{corollary}
\begin{proof}We call a triple $(F,g,x)$ \emph{admissible} if $(F(g^n x))$ is a bounded nilsequence on $G/\Gamma$ with  Lipschitz constant $M$.
Define for $\omega,N\in\N$ and admissible $(F,g,x)$
%bounded nilsequence $(F(g^n x))$ on $G/\Gamma$ with  Lipschitz constant $M$
$$
a_{\omega,(F,g,x)}(N):=\max_{r< W,(r,W)=1} \left|\aveN (\Lambda'_{r,\omega}(n)-1) F(g^n x)\right|
$$
and assume that the claimed uniform convergence does not hold. Then there exist $\veps>0$, a subsequence $(\omega_j)$ of $\N$ and a sequence of admissible $(F_j,g_j,x_j)$ so that
$$
\limsup_{N\to\infty} a_{\omega_j,(F_j,g_j,x_j)}(N)>\veps \quad \text{for all }j\in\N.
$$
In particular, there exists a subsequence $(N_j)$ of $\N$ such that $a_{\omega_j,(F_j,g_j,x_j)}(N_j)>\veps$ for every $j\in\N$.

Define now the function $\omega:\N\to\N$ by
$$
\omega(N):=\omega_j \quad \text{ if } N\in [N_j,N_{j+1})
$$
which grows sufficiently slowly if $(N_j)$ grows sufficiently fast.
Then we have
$$
a_{\omega(N_j),(F_j,g_j,x_j)}(N_j)=a_{\omega_j,(F_j,g_j,x_j)}(N_j)>\veps\quad \text{for all }j\in\N
$$
contradicting Theorem \ref{thm:GT} which states that 
$\lim_{N\to\infty} a_{\omega(N),(F,g,x)}(N)=0$ uniformly in admissible $(F,g,x)$.
% Note that this argument respects the claimed uniformity in $F,g$ and $x$.
\end{proof}

%%%%%%%%%%%
%         %
%  Proof  %
%         %
%%%%%%%%%%%

\section{Proof of Theorem \ref{thm:main}}

We first need several standard simple facts.
\begin{lemma}\label{lemma}\emph{(See, e.g., ~\cite{FHK07})}
For a bounded sequence $(a_n)\subset\C$ one has
$$
\lim_{N\to\infty}\left| \avepN a_p -\aveN \Lambda'(n)a_n\right|=0.
$$
\end{lemma}

\begin{lemma}\label{lemma:W-trick}
Let $(b_n)\subset \C$
satisfy $b_n=o(n)$.
Then the following assertions hold.
\begin{itemize}
\item[(a)] 
The sequence $(b_n)$ is Ces\`aro convergent if and only if for every $\veps>0$ there exist $W,N_0 \in\N $ such that 
\begin{equation}\label{eq:lemma-W}
\left| \aveWN b_n - \aveWM b_n\right|<\veps \quad \forall N,M\geq N_0.
\end{equation}
\item[(b)] If $(b_n)$ is supported on the primes, then  for every $W\in \N$
\begin{equation}\label{eq:W-trick}
\frac{1}{WN} \sum_{n=1}^{WN}b_n= \frac{1}{W} \sum_{r<W, (r,W)=1}  \aveN b_{Wn+r} +o_W(1).
\end{equation}
\end{itemize}
\end{lemma}
\begin{proof} 
(a) 
The ``only if'' implication is clear.
To show the ``if'' implication, let $\veps>0$ and take $W,N_0$  satisfying (\ref{eq:lemma-W}). Let further $N_1\in\N$ be such that $|b_n|<\frac{\veps n}W$ holds for every $n\geq N_1$.
%Further take $W$ and $N_0$ satisfying (\ref{eq:lemma-W}). 
We can assume without loss of generality that $WN_0\geq N_1$.
For $k\geq WN_0$ let $N=N(k)\geq N_0$ be such that $k\in [WN,W(N+1))$. By the triangle inequality it suffices to show 
\begin{equation}\label{eq:lemma-W-2}
\left| \avek b_n-\aveWN b_n \right|<4\veps
%(2W+1)\veps
\end{equation}
if $k$ is large enough.

We first observe that 
\begin{eqnarray*}
\left| \avek b_n-\frac1{WN}\sum_{n=1}^k b_n \right|&=&
\frac{k-WN}{kWN} \sum_{n=1}^k|b_n|
\leq \frac{W}{k(k-W)} \left(O(1)+\sum_{n=N_1}^k \frac{k\veps}W\right)\\
&\leq& o(1) + \frac{k}{k-W}\veps 
<2\veps
\end{eqnarray*}
for large enough $k$.
On the other hand, 
\begin{eqnarray*}
\left| \frac1{WN}\sum_{n=1}^k b_n- \aveWN b_n\right|\leq \frac1{WN}\sum_{n=WN+1}^k \frac{n\veps}W
\leq \frac{k}{k-W} \veps <2\veps
\end{eqnarray*}
for large enough $k$, proving 
%The triangle inequality implies 
(\ref{eq:lemma-W-2}).

(b) The growth condition implies
\begin{eqnarray*}
\frac{1}{WN} \sum_{n=1}^{WN} b_n
= \frac{1}{WN} \sum_{r=1}^W \sum_{n=0}^{N-1} b_{Wn+r}
=\frac{1}{W} \sum_{r=1}^W \aveN b_{Wn+r} + o_W(1).
\end{eqnarray*}
If $(b_n)$ is supported on the primes, (\ref{eq:W-trick}) follows.
\end{proof}

The following property of connected nilsystems is well known.

\begin{lemma}\label{lemma:ergodic}
Let $X:=G/\Gamma$ be a connected nilsystem with Haar measure $\mu$ and $g\in G$. Then $(X,\mu,g)$ is ergodic if and only if $(X,\mu,g)$ is totally ergodic.
\end{lemma}
\begin{proof}
Since ergodicity of a nilsystem is equivalent to ergodicity of its Kronecker factor (also called maximal factor-torus, or ``horizontal'' torus) $G/([G,G]\Gamma)$, see Leibman \cite{L}, we can assume without loss of generality  that $X$ is a compact connected abelian group.

Let $(X,\mu,g)$ be ergodic, $m\in \N$ and let $F\in L^2(X,\mu)$ be an  $g^m$-invariant function, i.e., $F(g^mx)=F(x)$ for every $x\in X$.  Consider the Fourier decomposition
$$
F=\sum_{\chi\in \hat{X}} c_\chi \chi.
$$
By the assumption we have
$$
F=\sum_{\chi\in \hat{X}} c_\chi (\chi(g))^m \chi.
$$
By the uniqueness of the decomposition we obtain
$$
c_\chi = c_\chi (\chi(g))^m \quad \forall \chi\in \hat{X}.
$$
Assume that $c_\chi\neq 0$. Then $(\chi(g))^m=1$, i.e.,  $\chi(g)$ is an $m^{\text{th}}$ root of unity. Since $(X,\mu,g)$ is ergodic,  $\{g^n: n\in \Z\}$ is dense in $X$. Since $\chi$ is a character and  $X$ is connected, $\chi(g)$ has to be equal to $1$ - otherwise $X$  would have two clopen components $\overline{\{g^{n}:\, m_0|n\}}$ and  $\overline{\{g^{n}:\, m_0\nmid n\}}$, where $m_0$ is the smallest  period of $\chi(g)$. Thus $F=c_11$ and $(X,\mu,g)$ is totally ergodic.
\end{proof}

\begin{proof}[Proof of Theorem \ref{thm:main}]

As mentioned above, we can assume that $x=\id_G\Gamma\in G^0$, where $G^0$ is the connected component of the identity in $G$, and $G=\langle G^0,g_1,\ldots,g_m\rangle$.

Every polynomial nilsequence can be represented as a linear nilsequence on a larger nilmanifold, see Leibman \cite[Prop. 3.14]{L}, Chu \cite[Prop. 2.1 and its proof]{C} and, in the context of connected groups, Green, Tao, Ziegler \cite[Prop. C.2]{GTZ}. Thus we can assume that $g(n)=g^n$ for some $g\in G$.

By the argument in Wooley, Ziegler \cite[text between Cor. 3.14 and Prop. 3.15]{WZ}, the nilsequence $(F(g^nx))$ can be written as a finite sum of (linear) nilsequences coming from a connected, simply connected Lie group.
Thus we can assume without loss of generality that $G$ is connected and simply connected.

We first assume that $F$ is Lipschitz and define $b_n:=\Lambda'(n) F(g^nx)$. To show convergence of
\begin{equation}\label{eq:ave-p}
\avepN F(g^px),
\end{equation}
by 
Lemma \ref{lemma} it is enough to show that $(b_n)$ satisfies the condition in Lemma \ref{lemma:W-trick}(a).

For every $\omega\in\N$ we have by Lemma \ref{lemma:W-trick}(b)
\begin{eqnarray}
\frac{1}{WN} \sum_{n=1}^{WN} b_n &=& \frac{1}{W} \sum_{r<W, (r,W)=1}  \aveN b_{Wn+r} + o_W(1)\nonumber \\
&=& \frac{1}{\phi(W)} \sum_{r<W, (r,W)=1} \aveN  \Lambda'_{r,\omega}(n) F(g^{Wn+r}x)+ o_W(1) \nonumber \\
&=& \frac{1}{\phi(W)} \sum_{r<W, (r,W)=1} \aveN  (\Lambda'_{r,\omega}(n)-1) F(g^{Wn+r}x) \nonumber \\
&\ & + \frac{1}{\phi(W)} \sum_{r<W, (r,W)=1}  \aveN  F(g^{Wn+r}x)+ o_W(1) \label{eq:proof}\\
&=:& I(N)+II(N)+o_W(1).\nonumber
\end{eqnarray}
Let $\veps>0$ and take a large $\omega$ such that $\limsup_{N\to\infty}|I(N)|<\veps$ which exists by Corollary \ref{cor}. Since the sequence $(F(g^{Wn+r}x))_{n\in\N}$ is Ces\`aro convergent for every $r$, see  Leibman \cite{L} and Parry \cite{P1,P2}, there is $N_0\in\N$ such that $|II(N)-II(M)|<\veps$ for every $N,M\geq N_0$. 
Thus $(b_n)$ satisfies the condition in Lemma \ref{lemma:W-trick}(a).

Take now $F \in C(G/\Gamma)$ arbitrary, $x\in G/\Gamma$ and $\veps> 0$. By the uniform continuity of $F$ there exists $G\in C(G/\Gamma)$ Lipschitz with $\|F-G\|_\infty\leq  \veps$.
We then have
\begin{eqnarray*}
&\ &\left|\avepN F(g^p x)-\avepM F(g^px)\right|\\
&\leq& \left|\avepN F(g^px)-\avepN G(g^px)\right|  \\
&\ & +\left|\avepN G(g^px)-\avepM G(g^px)\right| \\
&\ & + \left|\avepM G(g^px)-\avepM F(g^px)\right| \\
&\leq& 2\veps + \left|\avepN G(g^px)-\avepM G(g^px)\right|
\end{eqnarray*}
which is less than $3\veps$ for large enough $N,M$ by the above, finishing the argument.

The last assertion of the theorem follows analogously from the decomposition (\ref{eq:proof})  using  Lemma \ref{lemma:ergodic}, the fact that a nilsystem is ergodic if and only if it is uniquely ergodic, see Parry \cite{P1,P2}, and the uniform convergence of Birkhoff's ergodic averages to the space mean for uniquely ergodic systems. The last step (for non-Lipschitz functions) should be modified by showing that the difference $\avepN F(g^px)-\aveN F(g^nx)$ converges to zero.
\end{proof}

\end{document}